\documentclass[11pt,a4paper]{article}
\usepackage{amsmath, amsfonts, amsthm, latexsym, amscd}
\usepackage{amsmath}
\usepackage{xcolor,graphicx,hyperref}
\usepackage{amsfonts}
\usepackage{amssymb}
\usepackage[all]{xy}

\usepackage[norefs,nocites]{refcheck}

\allowdisplaybreaks[4]

\newtheorem{theorem}{Theorem}[section]
\newtheorem{proposition}[theorem]{Proposition}
\newtheorem{corollary}[theorem]{Corollary}
\newtheorem{example}[theorem]{Example}

\newtheorem{remark}[theorem]{Remark}

\newtheorem{lemma}[theorem]{Lemma}
\newtheorem{final remark}[theorem]{Final Remark}
\newtheorem{definition}[theorem]{Definition}

\begin{document}

\title{An injective-type norm and integral bilinear forms\\ defined by sequence classes}
\author{Jamilson R. Campos\,, Lucas Nascimento\thanks{Supported by a CAPES scholarship.}\, and Luiz Felipe P. Sousa \thanks{Supported by a CAPES scholarship. \newline 2020 Mathematics Subject Classification: 47L20, 46B45, 47B37, 46B10, 46A32, 47A07.\newline Keywords: Sequence spaces, tensor norms, operator ideals, integral bilinear forms.}}
\date{}
\maketitle

\begin{abstract}
In this work we define a class of injective-type norm on tensor products through the environment of sequence classes. Examples and results on this norm will be presented and the duality is studied in this context. As a byproduct, we present the definition of the associated integral-type bilinear forms and also a tensor characterization for a class of sequence spaces.
\end{abstract}

\section{Introduction}
The study of the duals $(E \hat{\otimes}_\alpha F)'$ (here $E$ and $F$ are Banach spaces and $\alpha$ is a reasonable crossnorm) as classes of linear operators or bilinear forms is a fundamental part in the theory of tensor products (see \cite{A.Defant, R. Ryan}) and establishes its close relationship with the parallel theory of operator ideals (see \cite{handbook,A.Pietsch} and \cite[6.3]{history}). Issues such as the maximality and minimality of ideals and the approximation property, for example, are involved in this scope and recent developments in these subjects can be seen in \cite{achour2018, sheldon, maite, maite2020, samuel, kim2020, kim,J. A. Lopez Molina.tres,turcovillafane}.

A number of important operator ideals are defined, or characterized, by the transformation of vector-valued sequences and a unifying approach to this kind of operators ideals was proposed in \cite{G. Botelho} using the concept of Sequence Classes (cf. Definition \ref{defseqcl}). For sequence classes $X$ and $Y$, a linear operator $T \colon E \longrightarrow F$ is $(X;Y)$-summing if $T((x_j))_{j=1}^\infty \in Y(F)$ whenever $(x_j)_{j=1}^\infty \in X(E)$ and the Banach operator ideal of such operators is denoted by ${\cal L}_{X;Y}$. This approach has shown to be quite fruitful (see \cite{complut, Jamilson.dual, espaco.mid, G. Botelho and D. Freitas, raquel, J.R.Campos.J.Santos, J. Ribeiro and F. Santos, J. Ribeiro and F. Santos.dois}) and the maximality of the Banach operator ideals of the type ${\cal L}_{X;Y}$ was recently studied, using a tensor approach, in \cite{BCN}.

It is well known that since the injective norm $\varepsilon$ on $E \otimes F$ is smaller than the projective norm $\pi$, every bounded linear functional on $E \hat{\otimes}_\varepsilon F$ is the linearization of a unique bounded bilinear form on $E \times F$. A complete description of the dual space of $E \hat{\otimes}_\varepsilon F$ leads to the definition of Integral Bilinear Forms and Integral Operators (see \cite[Section 3.4]{R. Ryan}). The purpose of this paper is to define and study a class of injective-type norms on tensor products, which generalizes the norm $\varepsilon$ in a certain sense, using the environment of sequence classes. We establish necessary results on these norms as well as present the duality theory for such normed tensor products. As an application, we present in the last section of the paper a tensor characterization for the space $(\ell_p^{\rm mid})^u(E)$.

Banach spaces over $\mathbb{K} = \mathbb{R}$ or $\mathbb{C}$ are denoted by $E$ and $F$. The symbol $E \stackrel{1}{\hookrightarrow} F$ means that $E$ is a linear subspace of $F$ and $\|x\|_F \leq \|x\|_E$ for every $x \in E$; and $E\stackrel{1}{=}F$ means that $E$ is isometrically isomorphic to $F$. We denote by $E'$ the topological dual of E, $B_E$ denotes the closed unit ball of $E$ and ${\cal L}(E; F)$ the Banach space of bounded linear operators from $E$ to $F$ with the usual operator norm. For $x \in E$ and $j \in \mathbb{N}$, the symbol $x\cdot e_j$ denotes the sequence $(0,\ldots, 0,x,0, 0,\ldots ) \in E^\mathbb{N}$, where $x$ is placed at the $j$-th coordinate. The symbol $(x_{j})_{j=1}^{n}$, where $x_1, \ldots, x_n \in E$, stands for the sequence $(x_{1},x_{2},\ldots,x_{n},0,0,\ldots) \in E^\mathbb{N}$. By $\widehat{T}$ we mean the induced operator $$\widehat{T} \colon E^\mathbb{N} \longrightarrow F^\mathbb{N}~,~\widehat{T}((x_j)_{j=1}^\infty) = (T(x_j))_{j=1}^\infty.$$ Restrictions of $\widehat{T}$ to subspaces of $E^{\mathbb{N}}$ are still denoted by $\widehat{T}$. The other notations and symbols used here are  usual in Functional Analysis.

The theory, nomenclature and symbology on the sequence class environment are present in \cite{G. Botelho} and the additional associated elements in \cite{Jamilson.dual,BCN,Ariel}. However, for the convenience of the reader, we will present below the definitions and essential tools on this topic.

\begin{definition}\rm \label{defseqcl}
According to \cite{G. Botelho}, a {\it Sequence Class} is a rule $X$ that assigns to each Banach space $E$ a Banach space $X(E)$ of $E$-valued sequences, that is $X(E)$ is a vector subspace of $E^{\mathbb{N}}$ with the coordinatewise operations, such that:\\
(i) $c_{00}(E) \subseteq X(E) \stackrel{1}{\hookrightarrow}  \ell_\infty(E)$ for every Banach space $E$,\\
(ii) $\|x \cdot e_j\|_{X(E)}= \|x\|_E$  for every Banach space $E$, every $x \in E$ and every $j \in \mathbb{N}$.
\end{definition}

To avoid ambiguities we shall eventually denote a sequence classe $X$ by $X(\cdot)$. For example, we write just $\ell_p^w$ instead of $\ell_p^w(\cdot)$, but we write $\ell_p(\cdot)$ so that there is no confusion with the scalar sequence space $\ell_p$.

Given sequence classes $X$ and $Y$, we say that an operator $T \in {\cal L}(E;F)$ is {\it $(X;Y)$-summing} if $T((x_j))_{j=1}^\infty \in Y(F)$ whenever $(x_j)_{j=1}^\infty \in X(E)$. In this case, the induced linear operator $\widehat{T} \colon X(E) \longrightarrow Y(F)$, is continuous and $\|T \|_{X;Y} : = \|\widehat{T}\|$ is a norm that makes the space ${\cal L}_{X;Y}(E;F)$ of $(X;Y)$-summing operators a Banach space. Whenever we refer to ${\cal L}_{X;Y}(E;F)$ we assume that it is endowed with the norm $\|\cdot \|_{X;Y}$. 

A sequence class $X$ is {\it linearly stable} if, regardless of the Banach spaces $E$ and $F$, every operator $T \in {\cal L}(E;F)$ is $(X;X)$-summming and $\|T\|_{X;X} = \|T\|$, that is, $\mathcal{L}_{X;X}(E;F)\stackrel{1}{=} \mathcal{L}(E;F)$. If the sequence classes $X$ and $Y$ are linearly stable and $X(\mathbb{K}) \stackrel{1}{\hookrightarrow} Y(\mathbb{K})$, then $\mathcal{L}_{X;Y}$ is a Banach operator ideal \cite[Theorem 3.6]{G. Botelho}.

We need to recall two more definitions from \cite{G. Botelho} and \cite{Jamilson.dual}, respectively. A sequence class $X$ is said to be:\\
\noindent (i) \emph{finitely determined} if for all sequence $(x_j)_{j=1}^\infty \in E^{\mathbb{N}}$, it follows that $(x_j)_{j=1}^\infty \in X(E)$ if and only if $\sup_k \|(x_j)_{j=1}^k  \|_{X(E)} < +\infty$ and, in this case, $\|(x_j)_{j=1}^\infty  \|_{X(E)} = \sup_k \|(x_j)_{j=1}^k  \|_{X(E)}. $\\
\noindent (ii) {\it spherically complete} if $(\lambda_jx_j)_{j=1}^\infty \in X(E)$ and $\|(\lambda_jx_j)_{j=1}^\infty \|_{X(E)} = \|(x_j)_{j=1}^\infty\|_{X(E)}$ whenever $(x_j)_{j=1}^\infty \in X(E)$ and $(\lambda_j)_{j=1}^\infty \in \mathbb{K}^\mathbb{N}$ is such that $|\lambda_j| = 1$ for every $j$.

\begin{example}\rm Let $1 \le p < \infty$ and $p^*$ defined by $1/p + 1/{p^*} = 1$, where we set $p^* = \infty$ if $p=1$. The following rules define spherically complete and linearly stable sequence classes, endowed with their usual norms: \\
$\bullet$ The class $E \mapsto c_0(E)$ of norm null sequences.\\
$\bullet$ The class $E \mapsto \ell_\infty(E)$ of bounded sequences.\\
$\bullet$ The class $E \mapsto \ell_p(E)$ of absolutely $p$-summable sequences.\\
$\bullet$ The class $E \mapsto \ell_p^w(E)$ of weakly $p$-summable sequences \cite[p.\,32]{J.Diestel}:
$$\ell_p^w(E) :=  \left\{(x_{j})_{j=1}^{\infty} \in E^{\mathbb{N}} : \|(x_j)_{j=1}^\infty\|_{w,p} = \sup_{\varphi \in B_{E'}}\|(\varphi(x_j))_{j=1}^\infty\|_p < \infty\right\}.$$
$\bullet$ The class $E \mapsto \ell_p^u(E)$ of unconditionally $p$-summable sequences \cite[Section 8.2]{A.Defant}:
$$\ell_p^u(E) :=  \left\{(x_{j})_{j=1}^{\infty} \in \ell_p^w(E) : \lim_{k \to \infty}\|(x_j)_{j=k}^\infty\|_{w,p} \to 0\right\},$$ with the norm inherited from $\ell_p^w(E)$.\\
\noindent $\bullet$ The class $E \mapsto \ell_p\langle E \rangle$ of Cohen strongly $p$-summable sequences \cite{cohen73}:
$$\ell_p\langle E \rangle :=  \left\{(x_{j})_{j=1}^{\infty} \in E^{\mathbb{N}} : \|(x_j)_{j=1}^\infty\|_{\ell_p\langle E \rangle}= \sup_{(\varphi_j)_{j=1}^\infty \in B_{\ell_{p^*}^w(E')}} \|(\varphi_j(x_j))_{j=1}^\infty\|_1 < \infty\right\}.$$

\noindent $\bullet$ The class $E \mapsto \ell_p^{\rm mid}(E)$  of mid $p$-summable sequences \cite{espaco.mid}:
\begin{equation}\label{defmid}
\ell_p^{\rm mid}(E) =  \left\{(x_{j})_{j=1}^{\infty} \in E^{\mathbb{N}} : \|(x_{j})_{j=1}^{\infty}\|_{p,{\rm mid}} = \sup\limits_{(\varphi_n)_{n=1}^\infty \in B_{\ell_p^w(E')}} \left(\sum_{n = 1}^\infty\sum_{j = 1}^\infty |\varphi_n(x_j)|^p \right)^{1/p} < \infty\right\}.
\end{equation}

\end{example}
\noindent All sequence classes presented above are finitely determined except $c_0(\cdot)$ and $\ell_p^u$.

The dual of a sequence class $X$ was introduced in \cite{Jamilson.dual} in the following fashion:
\begin{equation*}
	X^{\rm dual}(E) = \left\{(x_j)_{j=1}^\infty\ \mathrm{in\ } E: \sum_{j=1}^\infty \varphi_j(x_j)\ \mathrm{converges\ } \text{for every}\ (\varphi_j)_{j=1}^\infty\ \mathrm{in\ } X(E')\right\}.
\end{equation*}
If $X$ is a linearly stable and spherically complete sequence class, the expression
\begin{equation*}
\left\|(x_{j})_{j=1}^{\infty} \right\|_{X^{\rm dual}}:= \sup_{(\varphi_{j})_{j=1}^{\infty}\in B_{X(E')} }\sum_{j=1}^{\infty}\left|\varphi_{j}(x_{j}) \right| = \sup_{(\varphi_{j})_{j=1}^{\infty}\in B_{X(E')} }\left|\sum_{j=1}^{\infty}\varphi_{j}(x_{j}) \right|
\end{equation*}
makes $X^{\rm dual}(E)$ a Banach space and $X^{\rm dual}$ a finitely determined, linearly stable and spherically complete sequence class (see \cite[Proposition 2.5]{Jamilson.dual}).

\begin{example}\rm
For $1 \leq p \leq \infty$, we have $(\ell_p^w)^{\rm dual} = \ell_{p^*}\langle \cdot \rangle$ and $(\ell_p)^{\rm dual} = \ell_{p^*}( \cdot )$. In particular, $(\ell_\infty)^{\rm dual} = \ell_{1}( \cdot )$ \cite[Example 2.6(d)]{Jamilson.dual}, which is somewhat surprising because $(\ell_\infty)' \neq \ell_1$.
\end{example}

\section{Dual classes}

In \cite{Jamilson.dual}, conditions on a sequence class $X$ were established so that we have $X^{\rm dual}(E') \stackrel{1}{=}X(E)'$, by the application \begin{align}\label{opdual} J\colon X^{\rm dual}(E') \longrightarrow X(E)'~,~
	J\left((\varphi_j)_{j=1}^\infty\right)\left((x_j)_{j=1}^\infty\right) = \sum_{j=1}^\infty \varphi_j(x_j).
\end{align} 
More specifically, this holds (\cite[Theorem 2.10]{Jamilson.dual}) if $X$ is \emph{dual-representable}, i.e., $X$ is linearly stable, finitely determined, spherically complete, finitely injective and such that $c_{00}(E)$ is dense in $X(E)$. By the way, a sequence class $X$ is said to be {\it injective} if \begin{equation*}
	\|(x_j)_{j=1}^\infty\|_{X(E)} \leq \|(i(x_j))_{j=1}^\infty\|_{X(F)}
\end{equation*}
whenever $i \colon E \longrightarrow F$ is a metric injection and $(x_j)_{j=1}^\infty \in E^\mathbb{N}$ is a sequence that fulfills  $(i(x_j))_{j=1}^\infty \in X(F)$. If we have only $	\|(x_j)_{j=1}^k\|_{X(E)} \leq \|(i(x_j))_{j=1}^k\|_{X(F)}$ for $k \in \mathbb{N}$ and $x_1, \ldots, x_k \in E$, we say that $X$ is {\it finitely injective}. 

\begin{example}\rm
The sequence classes $c_0(\cdot), \ell_\infty(\cdot), \ell_p(\cdot), \ell_p^w$ and $\ell_p^u$ are injective.
\end{example}

Some facts presented below are immediate consequences of the involved definitions.

\begin{proposition}\label{injectivefacts} Let $X$ be a sequence class. It follows that:
	\begin{description}
		\item[a)] If $X$ is injective, then it is also finitely injective;
		\item[b)] If $X$ is finitely injective and finitely determined, then it is also injective;
		\item[c)] If $X$ is linearly stable and injective, then $\|\cdot\|_{X(G)} = \|\cdot\|_{X(E)} \ {\rm in }  \ X(G)$ regardless of the subspace $G$ of $E$.
	\end{description}
\end{proposition}

The next example shows that, even though it is not possible to use the result  \cite[Theorem 2.10]{Jamilson.dual} to obtain the duality $X^{\rm dual}(E') \stackrel{1}{=}X(E)'$, there may exist a sequence class $X'$ such that $X'(E') \stackrel{1}{=}X(E)'$.

\begin{example}\label{exnew}\rm For $1 < p < \infty$, the sequence class $X=\ell_{p}\langle \cdot \rangle$ is not injective (this is not difficult to prove and is done in \cite[Proposição 4.4.17]{Davidson}) and since it is finitely determined, from Proposition \ref{injectivefacts} b), the only option left is that it is not finitely injective too. So we cannot use the result \cite[Theorem 2.10]{Jamilson.dual} to obtain the isomorphism in \eqref{opdual} for the sequence class $X^{\rm dual}$. However, as ${\cal L}(\ell_{p^*};E) \stackrel{1}{=} \ell_{p}^w(E)$ (see \cite[Proposition 2.2]{J.Diestel}) and $\ell_{p}\langle E \rangle \stackrel{1}{=} E \widehat{\otimes}_\pi \ell_p$ (see \cite[Section 2]{Bu-Diestel}), we have
	\[\ell_{p^*}^w(E') \stackrel{1}{=} {\cal L}(\ell_{p};E') \stackrel{1}{=} (\ell_{p} \widehat{\otimes}_\pi E)' \stackrel{1}{=} (\ell_{p} \langle E \rangle)'  \]
where the second isomorphism comes from tensor product theory (see \cite[Section 2.2]{R. Ryan}).
\end{example}

Motivated by the above circumstance, we are invited to explore sequence classes of type $X'$ for which, as the dual $X^{\rm dual}$ of a class $X$, provide an isomorphism as in \eqref{opdual}. We will refer to such a sequence class as a \emph{dual class for $X$}.

\begin{definition}\rm
	Let $X$ be a sequence class. We say that a sequence class $X'$ is a \emph{dual class for} $X$ if $X'(E')\stackrel{1}{=}X(E)'$ for any Banach space $E$, by the application
	\begin{equation}\label{isodual}
		\Psi: X'(E')\longrightarrow X(E)'\ ,\ \Psi\left((\varphi_{j})_{j=1}^{\infty} \right) \left((x_{j})_{j=1}^{\infty} \right)  = \sum_{j=1}^{\infty}\varphi_{j}(x_{j}), 
	\end{equation}
	for all $(\varphi_{j})_{j=1}^{\infty} \in X'(E')$ and all $(x_{j})_{j=1}^{\infty} \in X(E)$.
\end{definition}

\begin{example}\label{exduals}\rm a) The sequence class $\ell_{1}(\cdot)$ is a dual class for $c_{0}(\cdot)$ and, for $1\leq p< \infty$, $\ell_{p^{\ast}}(\cdot)$ is a dual class for  $\ell_{p}(\cdot)$.\\
\noindent b) For $1\leq p< \infty$, $\ell_{p^{\ast}}^{w}$ is a dual class for $\ell_{p}\langle\cdot\rangle$ (cf. Example \ref{exnew}) and  $\ell_{p^{\ast}}\langle\cdot\rangle$ is a dual class for $\ell_{p}^{u}$ (see \cite[Proposition 6.]{bu2005}).\\
\noindent c) If $M$ is an  Orlicz function (see \cite[Chapter 4]{lt1}) such that $M(1)=1$, it is not difficult to show that the rules $E \mapsto \ell_M(E):=\left\{(x_j)_{j=1}^\infty  \in E^{\mathbb{N}} : (\|x_j\|)_{j=1}^\infty \in \ell_{M}\right\}$ and $E \mapsto h_M(E):=\left\{(x_j)_{j=1}^\infty  \in E^{\mathbb{N}} : (\|x_j\|)_{j=1}^\infty \in h_{M}\right\}$ define sequence classes and both are spherically complete, linearly stable, finitely determined and injective. Furthermore, if $M$ satisfies the $\Delta_2$-condition at zero, then  $\ell_{M^*}(\cdot)$ is a dual class for $\ell_{M}(\cdot)$, where  $M^*$ is the function complementary to $M$. We refer to \cite{Ariel,Gupta} for more information on Orlicz vector-valued sequence spaces. 
\end{example}

\begin{definition}\rm
	We say that a sequence class $X$ has the \emph{extension property} if for any Banach space $E$, every $M$ subspace of $E$ and each $(\varphi_{j})_{j=1}^{\infty}\in X(M')$, there exists $(\widetilde{\varphi}_{j})_{j=1}^{\infty}\in X(E')$ such that  $\|(\widetilde{\varphi}_{j})_{j=1}^{\infty} \|_{X(E')} \le \|(\varphi_{j})_{j=1}^{\infty} \|_{X(M')}$ and $\varphi_{j}(x)= \widetilde{\varphi}_{j}(x)$
	for all $x\in M$ and all $j\in \mathbb{N}$.
\end{definition}

\begin{example}\rm
The sequence classes $c_{0}(\cdot), h_M(\cdot), \ell_M(\cdot)$ and $\ell_{p}(\cdot)$, for $1\leq p\leq \infty$, have the extension property as an immediate consequence of the Hahn-Banach extension theorem. 
\end{example} 

\begin{proposition}
	If $X$ is a linearly stable and injective sequence class, then every dual class $X'$ for $X$ has the extension property.
\end{proposition}
\begin{proof} Let $M$ be a subspace of $E$ and $(\varphi_{j})_{j=1}^{\infty}\in X'(M')$. Our hypothesis and Proposition \ref{injectivefacts} c) allow us to take $X(M)$ as a normed subspace of $X(E)$ and we obtain the following diagram 
$$(\varphi_{j})_{j=1}^{\infty}\in X'(M') \longmapsto \varphi \in X(M)' \xrightarrow{\hspace*{1cm}} \widetilde{\varphi} \in X(E)' \longmapsto (\widetilde{\varphi}_{j})_{j=1}^{\infty}\in X'(E'), $$
where $\widetilde{\varphi}$ is a Hahn-Banach extension of $\varphi$. From this we have
$$\|(\varphi_j)^\infty_{j=1}\|_{X'(M')}=\|\varphi\|=\|\widetilde{\varphi}\| = \|(\widetilde{\varphi}_j)^\infty_{j=1}\|_{X'(E')}$$
and also
$$\varphi((x_j)^\infty_{j=1}) = \sum_{j=1}^\infty\varphi_j(x_j) \textrm{ and } \widetilde{\varphi}((x_j)^\infty_{j=1}) = \sum_{j=1}^\infty\widetilde{\varphi}_j(x_j),$$
for all $(x_j)^\infty_{j=1} \in X(M)$. In particular, for all $x \in M$ and for all $j \in \mathbb{N}$, we have $x \cdot e_j \in X(M)$ and so $\widetilde{\varphi}_j(x) = \widetilde{\varphi}(x \cdot e_j)=\varphi(x \cdot e_j) = \varphi_j(x).$
\end{proof}

\begin{example}\rm
The sequence class $\ell_p\langle \cdot \rangle$ has the extension property as a consequence of Example \ref{exduals} b) combined with the above proposition.
\end{example}

\section{Injective-type norms}

Let us now introduce a class of reasonable norms, involving sequence classes, on the tensor product $E\otimes F$. These norms are, from a certain point of view, generalizations of the injective norm $\varepsilon$. 

We need to use certain definitions and results on the scalar component $X(\mathbb{K})$ of a sequence class $X$ and so we present some of the required theory of scalar-valued sequence spaces. Next, for the convenience of the reader, we present some definitions in the scope of tensor norms.

\begin{definition}\rm \cite[Definition 3.1]{Ariel} A \textit{scalar sequence space} $\lambda$ is a Banach space formed by scalar-valued sequences, that is, $\lambda \subseteq \mathbb{K}^\mathbb{N}$, endowed with the usual coordinatewise algebraic operations, satisfying the following conditions:
	(i) $c_{00} \subseteq \lambda \stackrel{1}{\hookrightarrow} \ell_\infty$.
	(ii) $(e_j)_{j=1}^\infty$ is a  Schauder basis for $\lambda$.
	(iii)  $\|e_j\|_{\lambda}=1$ for every $j\in\mathbb{N}$.
\end{definition}
Defining also the sequence space
$\lambda_*:=\{ (\varphi(e_j))_{j=1}^\infty:\varphi\in \lambda^* \} \subseteq \mathbb{K}^\mathbb{N}$
we have the following

\begin{proposition} {\rm \cite[Proposition 3.3]{Ariel}} Let $\lambda$ be a scalar sequence space. Then $\lambda_*$ is a linear space with the usual algebraic operations, the map
	$$\|\cdot \|_{\lambda_*}\colon \lambda_* \to [0,\infty)~,~ \|(\varphi(e_j))_{j=1}^\infty \|_{\lambda_*}=\|\varphi\|, $$
	is a norm on $\lambda_*$ and the correspondence
	$\varphi \in \lambda' \mapsto (\varphi(e_j))_{j=1}^\infty \in \lambda_*$
	is an isometric isomorphism. In particular, $\lambda_*$ is a Banach space. Besides that, we have:\\
	{\rm a)} for every $(\alpha_j)_{j=1}^\infty\in \lambda$, $(\alpha_j)_{j=1}^\infty=\sum\limits_{j=1}^\infty \alpha_j e_j$ and {\rm b)} $c_{00} \subseteq \lambda_* \stackrel{1}{\hookrightarrow} \ell_\infty$ and $\|e_j\|_{\lambda_*}=1$, \ $\forall j\in\mathbb{N}$.
\end{proposition}

\begin{example}\rm The following are scalar sequence spaces:
	$\lambda = \ell_p$ with $\lambda_* = \ell_{p^*}$ for every $1 < p < \infty$, $\lambda = c_0$ with $\lambda_* = \ell_1$, $\lambda = \ell_1$ with $\lambda_* = \ell_\infty$, $\lambda = h_M$ with $\lambda_* = \ell_{M^*}$ (see \cite[Examples 3.2]{Ariel}) and $\lambda = \ell_M$ with $\lambda_* = \ell_{M^*}$ if $M$ satisfies the $\Delta_2$-condition at zero.
\end{example}

The norms $\varepsilon$ and $\pi$ on $E\otimes F$ are given by (see \cite[Chapters 2 and 3]{R. Ryan}) 
\[\varepsilon(u) = \sup\left\{\left|\sum_{j=1}^k \varphi(x_j) \psi(y_j)\right|: \varphi \in B_{E'}, \psi \in B_{F'}\right\},\]
where $\sum_{j=1}^k x_j\otimes y_j$ is any representation ou $u$, and
$\pi(u) = \inf\left\{\sum_{j=1}^k \|x_j\| \|y_j\|\right\}$,
the infimum being taken over all representations of $u$. We need to use the symbol $B(E,F)$ to represent the space of all bilinear forms on $E\times F$.

We say that a norm $\alpha$ on $E \otimes F$ is \emph{a reasonable crossnorm} if: (i) $\alpha(x \otimes y) \leq \|x\| \cdot \|u\|$ for all $x\in E$ and $y \in F$. (ii) The linear functional $\varphi \otimes \psi$ on $E \otimes F$ is continuous and $\|\varphi \otimes \psi\| \le \|\varphi\|\|\psi\|$ for all $\varphi \in E'$ and $\psi \in F'$. A characterization given in \cite[Proposition 6.1]{R. Ryan} ensures that $\alpha$  is a reasonable crossnorm if and only if $\varepsilon(u) \le \alpha(u) \le \pi(u)$ for every $u \in E \otimes F$.

We say that a correspondence that assigns, for all Banach spaces $E$ and $F$, a norm $\alpha$ on $E \otimes F$  is {\it a tensor norm} if:\\
$\bullet$ $\alpha$ is a reasonable crossnorm.\\
$\bullet$ $\alpha$ is \emph{uniform}, that is, for all Banach spaces $E_1,E_2, F_1, F_2$ and all operators $T_i \in {\cal L}(E_i,F_i)$, $i = 1,2$, holds $\|T_1 \otimes T_2 \colon E_1 \otimes_\alpha E_2 \longrightarrow F_1 \otimes_\alpha F_2\| \leq \|T_1\|\cdot \|T_2\|$.\\
$\bullet$ $\alpha$ is \emph{finitely generated}, that is, for all Banach spaces $E,F$ and any $u \in E \otimes F$,
$$\alpha(u; E \otimes F) = \inf\left\{\alpha(u; M \otimes N) : u \in M \otimes N, M \in {\cal F}(E),N \in {\cal F}(F) \right\}, $$
where ${\cal F}(E)$ is the collection of all finite dimensional subspaces of $E$.

We can now introduce our class of norms.

\begin{proposition}\label{normagxy}
	Let $E$ and $F$ be Banach spaces and $\lambda$ be a scalar sequence space. Then, the expression  
	$$\alpha_\lambda^{ }(u)= \sup\left\lbrace \left|\sum_{j=1}^{k}\sum_{n=1}^{\infty}\varphi_{n}(x_{j})\psi_{n}(y_{j}) \right|: (\varphi_{n})_{n=1}^{\infty}\in B_{X(E')}, (\psi_{n})_{n=1}^{\infty} \in B_{Y(F')} \right\rbrace$$
	defines a reasonable crossnorm on $E\otimes F$ for any linearly stable sequence classes $X$ and $Y$ that meet the conditions $X(\mathbb{K})\stackrel{1}{\hookrightarrow} \lambda$ and $Y(\mathbb{K})\stackrel{1}{\hookrightarrow} \lambda_*$. 
\end{proposition}
\begin{proof} For $x\in E$ and $y\in F$, we define the form 
	\begin{align*}
		B_{x,y}:& X(E')\times Y(F')\longrightarrow \mathbb{K}\\
		&((\varphi_{n})_{n=1}^{\infty}, (\psi_{n})_{n=1}^{\infty})\longmapsto B_{x,y}((\varphi_{n})_{n=1}^{\infty},(\psi_{n})_{n=1}^{\infty})= \sum_{n=1}^{\infty}\varphi_{n}(x)\psi_{n}(y). 
	\end{align*}
	Note that for any $x \in E$ the linear operator $T_x: E' \to \mathbb{K}$ given by $T_x(\varphi) = \varphi(x)$ is continuous and $\|T_x\|=\|x\|$. So, the linear stability of $X$ gives us 
	\begin{equation*}
	\|(\varphi_{n}(x))^\infty_{n=1}\|_{X(\mathbb{K})} = \|(T_x(\varphi_{n}))_{n=1}^\infty\|_{X(\mathbb{K})} \leq \|x\|\left\|(\varphi_{n})_{n=1}^{\infty} \right\|_{X(E')},
	\end{equation*}	
	for all $(\varphi_{n})_{n=1}^{\infty} \in X(E')$. Thus, as $X(\mathbb{K})\stackrel{1}{\hookrightarrow} \lambda$ and $Y(\mathbb{K})\stackrel{1}{\hookrightarrow} \lambda_*$, the application $B_{x,y}$ is well-defined since  	
	\begin{align}\label{contagamma}
		\left|\sum_{n=1}^{\infty}\varphi_{n}(x) \psi_{n}(y)\right|	&\leq \|(\varphi_{n}(x))^\infty_{n=1}\|_\lambda \|(\psi_{n}(y))^\infty_{n=1}\|_{\lambda_*} \\
		& \le \|(\varphi_{n}(x))^\infty_{n=1}\|_{X(\mathbb{K})} \|(\psi_{n}(y))^\infty_{n=1}\|_{Y(\mathbb{K})}\nonumber \\
		&\leq \|x\| \|y\| \left\|(\varphi_{n})_{n=1}^{\infty} \right\|_{X(E')} \left\|(\psi_{n})_{n=1}^{\infty} \right\|_{Y(F')} < \infty, \nonumber
	\end{align}
	where the first inequality comes from a H\"older-type inequality for scalar sequence spaces (see \cite[Lemma 3.9]{Ariel}). It easy to see that $B_{x,y}$ is bilinear and its continuity follows from \eqref{contagamma}. It is also easy to show that the application $A: E\times F\longrightarrow B( X(E'),Y(F')) $ given by $A(x,y)=B_{x,y}$ is bilinear and we can take its linearization 
	\[ A_L: E\otimes F \longrightarrow B( X(E'), Y(F'))\ , \ 
		u= \sum_{j=1}^{k}x_{j}\otimes y_{j}\longmapsto \sum_{j=1}^{k}B_{x_{j},y_{j}},\]
	for which  $\sum_{j=1}^{k}B_{x_{j},y_{j}}((\varphi_{n})_{n=1}^{\infty}, (\psi_{n})_{n=1}^{\infty})= \sum_{j=1}^{k}\sum_{n=1}^{\infty}\varphi_{n}(x_{j})\psi_{n}(y_{j})$. Following in a similar way to the calculus made in \eqref{contagamma}, we get
	\begin{align*}
		\left| A_L(u)((\varphi_{n})_{n=1}^{\infty}, (\psi_{n})_{n=1}^{\infty})\right| &\leq \sum_{j=1}^{k}\left|\sum_{n=1}^{\infty} \varphi_{n}(x_{j})\psi_{n}(y_{j})  \right| \\
		&\leq \left\|(\varphi_{n})_{n=1}^{\infty} \right\|_{X(E')} \left\|(\psi_{n})_{n=1}^{\infty} \right\|_{Y(F')}\sum_{j=1}^{k} \|x_{j}\| \|y_{j}\|,	
	\end{align*}
	that is, $A_L$ maps $E\otimes F$ into continuous bilinear forms.  A straightforward calculus shows that $A_L$ is injective and all this allows us to define a norm $\alpha_\lambda^{ }$ on $E\otimes F$ by setting	
	\begin{equation*}
		\alpha_\lambda^{ }(u):= \left\|A_L(u) \right\|
		=\sup\left\lbrace \left|\sum_{j=1}^{k}\sum_{n=1}^{\infty}\varphi_{n}(x_{j})\psi_{n}(y_{j}) \right|: (\varphi_{n})_{n=1}^{\infty}\in B_{X(E')}, (\psi_{n})_{n=1}^{\infty} \in B_{Y(F')} \right\rbrace.
	\end{equation*}	
	It is immediate that $\varepsilon\leq \alpha_\lambda^{ }$ and, reasoning as in \eqref{contagamma}, we obtain 
	$$\alpha_\lambda^{ }(x\otimes y)= \sup\left\lbrace \left|\sum_{n=1}^{\infty}\varphi_{n}(x)\psi_{n}(y) \right|: (\varphi_{n})_{n=1}^{\infty}\in B_{X(E')}, (\psi_{n})_{n=1}^{\infty}\in B_{Y(F')}  \right\rbrace \leq \left\| x\right\| \left\| y\right\|.$$
	This and the fact that $\alpha_\lambda^{ }$ satisfies the triangle inequality by its definition yield that $\alpha_\lambda^{ }\leq \pi$. Thus $\alpha_\lambda^{ }$ is a reasonable crossnorm on $E\otimes F$.
\end{proof}

The symbol $E\otimes_{\alpha_\lambda^{ }}F$ denotes the tensor product $E\otimes F$ endowed with the norm $\alpha_\lambda^{ }$ and its completion will be denoted by $E\widehat{\otimes}_{\alpha_\lambda^{ }}F$. 

\begin{remark}\rm a) In Propositions \ref{normagxy} the roles of $\lambda$ and $\lambda_*$ can be exchanged by adjusting the associated hypotheses without compromising the results obtained.\\
b) Observing the proof of Proposition \ref{normagxy} we realize that, with the same hypotheses and with an argument similar to the one developed in \eqref{contagamma}, we can prove (and we did this, but that is not the scope of this paper) the inequality $\varepsilon \le \alpha_{X,Y}$, a formerly imposed condition in \cite[Proposition 2.2]{BCN}.\\
c) In all results presented from now on, whenever we assert that $\alpha_\lambda^{ }$ is a reasonable norm and/or enjoy a certain property, we will be admitting that any sequence classes $X$ and $Y$ in the definitions of this norm satisfy sufficient conditions for this assertion. 
\end{remark}

\begin{proposition}
	The norm $\alpha_\lambda^{ }$ is uniform regardless of the  scalar sequence space $\lambda$. 
\end{proposition} 

\begin{proof} In what follows the operator $T': F' \to E'$ is the adjoint operator of $T \in \mathcal{L}(E;F)$.
	
	Let $E_{i}$ and $F_{i}$ be Banach spaces and $T_{i}\in \mathcal{L}(E_{i}; F_{i}), i=1,2$. Consider the linear operator $T_{1}\otimes T_{2}: E_{1}\otimes_{\alpha_\lambda^{ }}E_{2}\longrightarrow F_{1}\otimes_{\alpha_\lambda^{ }}F_{2}$ given by $T_{1}\otimes T_{2}(x\otimes y)= T_{1}(x)\otimes T_{2}(y)$. Since $X$ and $Y$ are linearly stable, we get
	$$\left\|(\varphi_{n} \circ T_{1})_{n=1}^{\infty} \right\|_{X(E_{1}')} = \left\|(T_{1}'(\varphi_{n}))_{n=1}^{\infty} \right\|_{X(E_{1}')}\leq \left\| T_{1}\right\|\left\|(\varphi_{n})_{n=1}^{\infty} \right\|_{X(F_{1}')}$$ whenever $(\varphi_{n})_{j=1}^{\infty} \in X(F_{1}')$, and the same can be done for $Y$ with the operator $T_2$. So, for $u\in E_1\otimes_{\alpha_\lambda^{ }}E_2$, we have
	\begin{align*}
		\alpha_\lambda^{ }&\left(T_{1}\otimes T_{2}(u), F_1 \otimes F_2\right) =\\
		&= \sup\left\lbrace \left|\sum_{j=1}^{k}\sum_{n=1}^{\infty}(\varphi_{n} \circ T_{1})(x_{j})(\psi_{n} \circ T_{2})(y_{j}) \right|: (\varphi_{n})_{n=1}^{\infty}\in B_{X(F_{1}')}, (\psi_{n})_{n=1}^{\infty}\in B_{Y(F_{2}')}  \right\rbrace  \\
		&= \left\|T_{1} \right\|\left\|T_{2} \right\|\cdot\\
		& \hspace{0.5cm} \sup\left\lbrace \left|\sum_{j=1}^{k}\sum_{n=1}^{\infty}\frac{T_{1}'}{\left\|T_{1} \right\| }(\varphi_{n})(x_{j})\frac{T_{2}'}{\left\|T_{2} \right\| }(\psi_{n})(y_{j}) \right|: (\varphi_{n})_{n=1}^{\infty}\in B_{X(F_{1}')}, (\psi_{n})_{n=1}^{\infty}\in B_{Y(F_{2}')}  \right\rbrace \\
		&\leq \left\| T_{1}\right\| \left\| T_{2}\right\| \sup\left\lbrace \left|\sum_{j=1}^{k}\sum_{n=1}^{\infty}\overline{\varphi}_{n}(x_{j})\overline{\psi}_{n}(y_{j}) \right|: (\overline{\varphi}_{n})_{n=1}^{\infty}\in B_{X(E_{1}')}, (\overline{\psi}_{n})_{n=1}^{\infty}\in B_{Y(E_{2}')}  \right\rbrace\\
		&= \left\| T_{1}\right\| \left\| T_{2}\right\| \alpha_\lambda^{ }(u, E_1 \otimes E_2)
	\end{align*}
	and therefore $T_{1}\otimes T_{2}$ is continuous and $\left\|T_{1}\otimes T_{2} \right\| \leq \left\|T_{1} \right\| \left\|T_{2} \right\| .$
\end{proof}

We say that a uniform reasonable crossnorm $\alpha$ is \emph{injective} (\emph{or respects subspaces}) if the norm induced on $M\otimes N$ by the norm of $E\otimes_{\alpha}F$ coincides with the norm on $M\otimes_{\alpha} N$ for all subspaces $M$ of $E$ and $N$ of $F$. It is well known that every injective uniform crossnorm is finitely generated (see \cite[Section 6.1]{R. Ryan}).

\begin{proposition} Let $\lambda$ be a scalar sequence space. Then $\alpha_\lambda^{ }$ is injective if $X$ and $Y$ have the extension property.
\end{proposition}

\begin{proof}
	Let $E$ and $F$ be Banach spaces with $M$ and $N$ subspaces of $E$ and $F$, respectively. Since $\alpha_\lambda^{ }$ is uniform, we get 
	\begin{equation*}
		\alpha_\lambda^{ }(u;E\otimes F )\leq \alpha_\lambda^{ }(u; M\otimes N),
	\end{equation*} for all $u\in M\otimes N$.	
	On the other hand, let $u\in M\otimes N$ and $\sum_{j=1}^{k}x_{j}\otimes y_{j}$ be a representation of $u$. Thus, as $X$ and $Y$ have the extension property, we get
	\begin{align*}
		\alpha_\lambda^{ }(u; M\otimes N)&= \sup\left\lbrace \left|\sum_{j=1}^{k}\sum_{n=1}^{\infty}\varphi_{n}(x_{j})\psi_{n}(y_{j}) \right|: (\varphi_{n})_{n=1}^{\infty}\in B_{X(M')}, (\psi_{n})_{n=1}^{\infty}\in B_{Y(N')}  \right\rbrace \\
		&= \sup\left\lbrace \left|\sum_{j=1}^{k}\sum_{n=1}^{\infty}\widetilde{\varphi}_{n}(x_{j})\widetilde{\psi}_{n}(y_{j}) \right|: (\widetilde{\varphi}_{n})_{n=1}^{\infty}\in B_{X(E')}, (\widetilde{\psi}_{n})_{n=1}^{\infty}\in B_{Y(F')}  \right\rbrace\\
		&\leq \sup\left\lbrace \left|\sum_{j=1}^{k}\sum_{n=1}^{\infty}\xi_{n}(x_{j})\zeta_{n}(y_{j}) \right|: (\xi_{n})_{n=1}^{\infty}\in B_{X(E')}, (\zeta_{n})_{n=1}^{\infty}\in B_{Y(F')}  \right\rbrace\\
		&= \alpha_\lambda^{ }(u; E\otimes F)
	\end{align*}
	and therefore $\alpha_\lambda^{ }$ is injective. 
\end{proof}

Compiling all the facts, examples and results presented, some norms $\alpha_\lambda^{ }$  are given in the following

\begin{example}\label{exnormas}\rm	Let $M$ and $M^*$ be complementary Orlicz functions (cf. Example \ref{exduals}) and let $1 < r,q < \infty$.\\
a) If $1<p<\infty$, $\lambda = \ell_p$, $r\leq p$ and $q\leq p^*$, then $\alpha_\lambda^{ }$ is a uniform reasonable norm for many pairs of presented sequence classes such as $$(X,Y) \in \{(\ell_r(\cdot),\, \ell_{q}^{\rm mid}),\, (\ell_r^{\rm mid},\, \ell_{q}\langle \cdot \rangle),\, (\ell_{r}\langle \cdot \rangle,\, \ell_{q}^w),\, (\ell_r^u,\, \ell_{q}\langle \cdot \rangle),\, (\ell_r^w,\, \ell_{q}(\cdot)),\, (\ell_r^u,\, \ell_{q}^{\rm mid})\}$$ and a tensor norm for  $(X,Y) \in \{(\ell_r(\cdot),\, \ell_{q}(\cdot)),\, (\ell_r(\cdot),\, \ell_{q}\langle \cdot \rangle)\}$.\\
b) For the case $p=1$, if we take any sequence class $X$ such that $X(\mathbb{K}) =\ell_1$, then $\alpha_\lambda^{ }$ is a uniform reasonable norm for all sequence class $Y$. For instance, $\alpha_\lambda^{ }$ is a uniform reasonable norm for all \[(X,Y) \in \{(\ell_1^w,\, \ell_\infty(\cdot)),\, (\ell_1^{\rm mid},\, c_0(\cdot)),\, (\ell_1^{\rm mid},\, \ell_p^w(\cdot)),\, (\ell_1^u(\cdot),\, \ell_M(\cdot))\}.\] 
c) If we take $X=\ell_1(\cdot)$ and take $Y$ with the extension property, then $\alpha_\lambda^{ }$ is a tensor norm. For instance, $\alpha_\lambda^{ }$ is a tensor norm for every \[(X,Y) \in \{(\ell_1(\cdot),\, \ell_\infty(\cdot)),\, (\ell_1(\cdot),\, \ell_p\langle \cdot \rangle),\, (\ell_1(\cdot),\, c_0(\cdot)),\, (\ell_1(\cdot),\, \ell_M(\cdot))\}.\]
d) If $\lambda = h_{M^*}$, then $\alpha_\lambda^{ }$ is a tensor norm for $X = h_{M^*}(\cdot)$ and $Y=\ell_M(\cdot)$. If $\lambda = \ell_M$ and $M$ satisfies the $\Delta_2$-condition at zero, then $\alpha_\lambda^{ }$ is a tensor norm for $X = \ell_M(\cdot)$ and $Y = \ell_{M^*}(\cdot)$.\\ 
\end{example}

Now let us consider a different interpretation of the definition of norm $\alpha_\lambda^{ }$ in terms of dual classes. To motivate this point of view, we will use a particular case present in Example \ref{exnormas} a).

For $1<p<\infty$ and $\lambda = \ell_p$, the norm
\[\alpha_\lambda^{ }(u)= \sup\left\lbrace \left|\sum_{j=1}^{k}\sum_{n=1}^{\infty}\varphi_{n}(x_{j})\psi_{n}(y_{j}) \right|: (\varphi_{n})_{n=1}^{\infty}\in B_{\ell_p(E')}, (\psi_{n})_{n=1}^{\infty} \in B_{\ell_{p^*}\langle F' \rangle} \right\rbrace\]
is a tensor norm (could be just uniform, depending on the choice of sequence classes). Since $X' =\ell_{p}(\cdot)$ is a dual class for $\ell_{p^*}(\cdot)$ and $Y' = \ell_{p^*}\langle \cdot \rangle$ is a dual class for $\ell_p^u$, we can write $\alpha_\lambda^{ }$ in the form
\[\alpha_\lambda'(u)= \sup\left\lbrace \left|\sum_{j=1}^{k}\sum_{n=1}^{\infty}\varphi_{n}(x_{j})\psi_{n}(y_{j}) \right|: (\varphi_{n})_{n=1}^{\infty}\in B_{X'(E')}, (\psi_{n})_{n=1}^{\infty} \in B_{Y'(F')} \right\rbrace.\]
Of course, the sequence class $X^{\rm dual}$ is part of the game and can be used to build $\alpha_\lambda'$ norms if $X$ is dual-representable. It is also clear that not every norm $\alpha_\lambda^{ }$ can be interpreted as a norm $\alpha_\lambda'$, since not every sequence class is necessarily dual for another, as for example $c_0(\cdot)$.

The advantage of this approach will become clear in the following section of the paper and lies in the fact that we can use the notation $\varphi=(\varphi_{n})_{n=1}^{\infty}$ to indicate that $\varphi\in X(E)'$, $(\varphi_{n})_{n=1}^{\infty}\in X'(E')$ and $(\varphi_{n})_{n=1}^{\infty}\longmapsto \varphi$ by the isometric isomorphism $X'(E')\stackrel{1}{=}X(E)'$.

\section{The dual of the space $E\widehat{\otimes}_{\alpha_\lambda'}F$}

In the context of the $\alpha_\lambda'$ norm, which is central to the results of this section, there always exist sequence classes $X$ and $Y$ for which the sequence classes $X'$ and $Y'$ (that define $\alpha_\lambda'$) are dual for $X$ and $Y$, respectively.

The following lemma shows that $ E\widehat{\otimes}_{\alpha_\lambda'}F$ can be identified as a subspace of the Banach space $\left(C(B_{X(E)'}\times B_{Y(F)'}), \left\|\cdot \right\|_{\infty}  \right)$, where $B_{X(E)'}$ and $B_{Y(F)'}$ carry their weak$^*$ topology.  

\begin{lemma}\label{lema.identificacao}
	Let $E$ and $F$ be Banach spaces and $\lambda$ be a scalar sequence space. The map 
	\begin{equation*}
		\varPhi: E\otimes_{\alpha_\lambda'}F\longrightarrow C(B_{X(E)'}\times B_{Y(F)'})\ ,\ 
		u = \sum_{j=1}^{k}x_{j}\otimes y_{j}\longmapsto \varPhi_{u},
	\end{equation*}
	where $\varPhi_{u}$ is given by $\varPhi_{u}(\varphi,\psi)= \displaystyle\sum_{j=1}^{k}\sum_{n=1}^{\infty}\varphi_{n}(x_{j})\psi_{n}(y_{j})$, is a linear isometry. In particular, its extension $\widetilde{\varPhi}: E\widehat{\otimes}_{\alpha_\lambda'}F\longrightarrow C(B_{X(E)'}\times B_{Y(F)'})$ is also a embedding.
\end{lemma}

\begin{proof}	
	Let us first show that $\varPhi$ is well-defined. Let $(\varphi^{\gamma}, \psi^{\gamma})_{\gamma\in \Gamma}$ a net in $B_{X(E)'}\times B_{Y(F)'}$ such that $(\varphi^{\gamma}, \psi^{\gamma})\stackrel{\gamma}{\longrightarrow}(\varphi, \psi) \in B_{X(E)'}\times B_{Y(F)'}.$ Since the projections are continuous, we have $\varphi^{\gamma}\stackrel{\gamma}{\longrightarrow}\varphi$ and $\psi^{\gamma}\stackrel{\gamma}{\longrightarrow}\psi$ in the respective weak$^*$ topologies and thus \begin{equation}\label{convest}\varphi^{\gamma}((x_{j})_{j=1}^{\infty})\stackrel{\gamma}{\longrightarrow} \varphi((x_{j})_{j=1}^{\infty}) \text{\ \ \  and \ \  } \psi^{\gamma}((y_{j})_{j=1}^{\infty})\stackrel{\gamma}{\longrightarrow} \psi((y_{j})_{j=1}^{\infty})
	\end{equation}
	for all $((x_{j})_{j=1}^{\infty}, (y_{j})_{j=1}^{\infty})\in X(E)\times Y(F)$.
	
	For each $\gamma\in \Gamma$, we can write $\varphi^{\gamma}=(\varphi_{n}^{\gamma})_{n=1}^{\infty}\in X'(E')$,  $\psi^{\gamma}=(\psi_{n}^{\gamma})_{n=1}^{\infty}\in Y'(F')$ and also  $\varphi=(\varphi_{n})_{n=1}^{\infty}\in X'(E')$ and $\psi= (\psi_{n})_{n=1}^{\infty}\in Y'(F')$. From the isomorphism in \eqref{isodual} and by \eqref{convest}, we get
	$$\varphi_{n}^{\gamma}(x)= \varphi^{\gamma}( x \cdot e_{n})\stackrel{\gamma}{\longrightarrow} \varphi( x \cdot e_{n})= \varphi_{n}(x) \ \text{ and }  \ \psi_{n}^{\gamma}(y)= \psi^{\gamma}( y \cdot e_{n})\stackrel{\gamma}{\longrightarrow} \psi( y\cdot e_{n})= \psi_{n}(y)$$
	for all $x\in E$, $y\in F$ and all $n\in \mathbb{N}$. Therefore,  
	\begin{equation}\label{conv.rede}
		\sum_{j=1}^{k}\sum_{n=1}^{l}\varphi_{n}^{\gamma}(x_{j})\psi_{n}^{\gamma}(y_{j})\stackrel{\gamma}{\longrightarrow} \sum_{j=1}^{k}\sum_{n=1}^{l}\varphi_{n}(x_{j})\psi_{n}(y_{j}),
	\end{equation}
	for all $l\in \mathbb{N}$. On the other hand, by the well-definition of the norm $\alpha_\lambda'$, we have
	\begin{equation}\label{conv.seq.}
		\lim_{l\rightarrow \infty} \sum_{j=1}^{k} \sum_{n=1}^{l}\varphi_{n}^{\gamma}(x_{j})\psi_{n}^{\gamma}(y_{j})=\sum_{j=1}^{k}\sum_{n=1}^{\infty}\varphi_{n}^{\gamma}(x_{j})\psi_{n}^{\gamma}(y_{j}) \ \text{for all} \ \gamma\in \Gamma,
	\end{equation}
	and
	\begin{equation}\label{conv.seq.2}
		\lim_{l\rightarrow \infty} \sum_{j=1}^{k} \sum_{n=1}^{l}\varphi_{n}(x_{j})\psi_{n}(y_{j})=\sum_{j=1}^{k}\sum_{n=1}^{\infty}\varphi_{n}(x_{j})\psi_{n}(y_{j}).
	\end{equation}
	Using \eqref{conv.rede} and \eqref{conv.seq.2} it is not difficult to show that  
	\begin{equation}\label{conv.rededupla}
		\sum_{j=1}^{k}\sum_{n=1}^{l}\varphi_{n}^{\gamma}(x_{j})\psi_{n}^{\gamma}(y_{j})\stackrel{(l,\gamma)}{\longrightarrow}\sum_{j=1}^{k}\sum_{n=1}^{\infty}\varphi_{n}(x_{j})\psi_{n}(y_{j})
	\end{equation}
	and from \eqref{conv.rede}, \eqref{conv.seq.} and \eqref{conv.rededupla} it follows that the iterated limits are the same and
	$$\lim\limits_{l\rightarrow \infty}\lim\limits_{\gamma}\sum_{j=1}^{k}\sum_{n=1}^{l}\varphi_{n}^{\gamma}(x_{j})\psi_{n}^{\gamma}(y_{j})= \sum_{j=1}^{k}\sum_{n=1}^{\infty}\varphi_{n}(x_{j})\psi_{n}(y_{j}) = \lim\limits_{\gamma}\lim\limits_{l\rightarrow \infty}\sum_{j=1}^{k}\sum_{n=1}^{l}\varphi_{n}^{\gamma}(x_{j})\psi_{n}^{\gamma}(y_{j}).$$
	Therefore,
	\begin{align*}
		\lim\limits_{\gamma}\varPhi_{u}(\varphi^{\gamma},\psi^{\gamma})&=\lim\limits_{\gamma}\sum_{j=1}^{k}\sum_{n=1}^{\infty}\varphi_{n}^{\gamma}(x_{j})\psi_{n}^{\gamma}(y_{j})=  \lim\limits_{\gamma}\lim_{l\rightarrow \infty} \sum_{j=1}^{k} \sum_{n=1}^{l}\varphi_{n}^{\gamma}(x_{j})\psi_{n}^{\gamma}(y_{j})\\
		& = \sum_{j=1}^{k}\sum_{n=1}^{\infty}\varphi_{n}(x_{j})\psi_{n}(y_{j})= \varPhi_{u}(\varphi, \psi)
	\end{align*}
	and this shows us that $\varPhi_{u}$ is continuous and $\varPhi$ is well-defined. Let us now prove that $\varPhi$ is linear. If $u,v\in E\otimes F$, $\beta\in \mathbb{K}$, $\varphi\in B_{X(E)'}$ and $\psi \in B_{Y(F)'}$, then
	\begin{equation*}
		\varPhi_{\beta u + v}(\varphi,\psi) = A_L(\beta u +v)(\varphi, \psi)= \beta A_L(u)(\varphi,\psi) + A_L(v)(\varphi,\psi) = \beta \varPhi_{u}(\varphi,\psi)+ \varPhi_{v}(\varphi,\psi),
	\end{equation*}
	where $A_L$ is the linearization of $A: E\times F\longrightarrow B( X'(E'), Y'(F')) $ given by $A(x,y)=B_{x,y}$ (cf. Proposition \ref{normagxy}).
	Therefore $\Phi$ is linear and it is immediate that $\left\|\varPhi_{u} \right\|_{\infty}= \alpha_\lambda'(u) $ for all $u\in E\otimes F$.
\end{proof}

Now we are in a position to characterize all continuous linear functional on $E\widehat{\otimes}_{\alpha_\lambda'}F$.

\begin{theorem}\label{teo.dual.gxy} 
	Let $E$ and $F$ be Banach spaces, $B:E \times F\longrightarrow \mathbb{K}$ a bilinear form and $\lambda$ a scalar sequence space. Then, $B_L:E \widehat{\otimes}_{\alpha_\lambda'}F\longrightarrow \mathbb{K}$ is a continuous linear functional if and only if there exists a regular Borel measure $\mu$ on the compact $B_{X(E)'}\times B_{Y(F)'}$ such that
	\begin{equation*}
		B(x,y)= \int_{B_{X(E)'}\times B_{Y(F)'}}\sum_{n=1}^{\infty}\varphi_{n}(x)\psi_{n}(y) d\mu(\varphi,\psi)	
	\end{equation*}
	for all $x\in E$ and $y\in F$. Furthermore, $\left\|B_L \right\|= \inf\left\|\mu \right\|$, where $\mu$ ranges over the set of all measures that correspond to $B$ in this way, and this infimum is attained.
\end{theorem}

\begin{proof}
	Suppose that $B_L:E \widehat{\otimes}_{\alpha_\lambda'}F\longrightarrow \mathbb{K}$ is a continuous linear functional. From the Lemma \ref{lema.identificacao}, we can identify $E \widehat{\otimes}_{\alpha_\lambda'}F$ as a subspace of $C(B_{X(E)'}\times B_{Y(F)'})$ and the Hahn-Banach Theorem ensures the existence of a continuous linear functional $\widetilde{B_L}:C(B_{X(E)'}\times B_{Y(F)'})\longrightarrow \mathbb{K}$ that extends $B_L$ and satisfies $\left\|B_L \right\|= \|\widetilde{B_L} \|.$
	
	By the Riesz Representation Theorem, the functional $\widetilde{B_L}$ is given by a single regular complex Borel measure $\mu$ in the form 
	$$\widetilde{B_L}(f)= \int_{B_{X(E)'}\times B_{Y(F)'}}f(\varphi,\psi) \ d\mu(\varphi,\psi),$$
	for each $f\in C(B_{X(E)'}\times B_{Y(F)'})$ and $\|\widetilde{B_L} \|= \left\| \mu\right\|.$  In particular, using Lemma \ref{lema.identificacao} for $f=x\otimes y$ we get
	$$B(x,y)=B_L(x\otimes y)= \widetilde{B_L}(x\otimes y)= \int_{B_{X(E)'}\times B_{Y(F)'}}\sum_{n=1}^{\infty}\varphi_{n}(x)\psi_{n}(y)d\mu(\varphi,\psi)$$
	for all $x\in E$ and $y\in F.$
	
	Conversely, since $B_L$ is the linearization of $B$, if $u\in E\otimes F$ and $\sum_{j=1}^{k}x_{j}\otimes y_{j}$ is a representation of $u$, we have 
	\begin{align*}
		B_L\left(u \right)&= \sum_{j=1}^{k}B_L(x_{j}\otimes y_{j})= \sum_{j=1}^{k}B(x_{j},y_{j})
		= \sum_{j=1}^{k}\int_{B_{X(E)'}\times B_{Y(F)'}}\sum_{n=1}^{\infty}\varphi_{n}(x_{j})\psi_{n}(y_{j})d\mu(\varphi,\psi)\\
		&= \int_{B_{X(E)'}\times B_{Y(F)'}}\sum_{j=1}^{k}\sum_{n=1}^{\infty}\varphi_{n}(x_{j})\psi_{n}(y_{j})d\mu(\varphi,\psi).
	\end{align*}
	From Lemma \ref{lema.identificacao} we known that $u=\Phi_u \in C(B_{X(E)'}\times B_{Y(F)'})$  and so $u$ is a Borel-measurable function. Moreover, since $\mu$ is a complex measure then $|\mu|$ is a positive finite measure and so $u\in L_{1}(|\mu|)$ because 
	\begin{align*}
		\int_{B_{X(E)'}\times B_{Y(F)'}}\left|u \right| \ d|\mu|(\varphi,\psi)
		&=   \int_{B_{X(E)'}\times B_{Y(F)'}}\left|\sum_{j=1}^{k}\sum_{n=1}^{\infty}\varphi_{n}(x_{j})\psi_{n}(y_{j})\right|d|\mu|(\varphi,\psi) \\
		&\leq \|\Phi_u\|_\infty \int_{B_{X(E)'}\times B_{Y(F)'}}d|\mu|(\varphi,\psi)\\
		&=  \alpha_\lambda'(u)\left| \mu\right| \left( B_{X(E)'}\times B_{Y(F)'}\right)
		= \alpha_\lambda'(u)\left\|\mu \right\|< \infty.
	\end{align*}
	Thus, from the Radon-Nikodym Theorem, there exists a $|\mu|$-integrable function $h$ on $B_{X(E)'}\times B_{Y(F)'}$ such that $|h|=1$ and 
	$$\int_{B_{X(E)'\times Y(F)'}}u \ d\mu(\varphi,\psi) = \int_{B_{X(E)'\times Y(F)'}}uh \ d|\mu|(\varphi,\psi),$$  	
	which brings us to
	\begin{align*}
		\left|B_L(u) \right|&= \left| \int_{B_{X(E)'}\times B_{Y(F)'}}\sum_{j=1}^{k}\sum_{n=1}^{\infty}\varphi_{n}(x_{j})\psi_{n}(y_{j})d\mu(\varphi,\psi)\right|\\
		&= \left|\int_{B_{X(E)'}\times B_{Y(F)'}}u  \ d\mu(\varphi,\psi)\right|= \left|\int_{B_{X(E)'}\times B_{Y(F)'}}uh  \ d|\mu|(\varphi,\psi)\right|  \\
		&\leq \int_{B_{X(E)'}\times B_{Y(F)'}}\left|u \right| \ d|\mu|(\varphi,\psi)
		\leq \alpha_\lambda'(u)\left\|\mu \right\|
	\end{align*}
	and therefore $B_L:E \otimes_{\alpha_\lambda'}F\longrightarrow\mathbb{K}$ is continuous. It follows that the extension $B_L:E \widehat{\otimes}_{\alpha_\lambda'}F\longrightarrow\mathbb{K}$ (here we use the same symbol) is continuous and $\left\|B_L\right\|\leq \left\|\mu \right\|.$ Finally, from what we did now and in the previous implication, we have $\left\|B_L\right\| = \inf\left\| \mu\right\| $ and that this infimum is attained.
\end{proof}

We have just characterized the elements of $(E\widehat{\otimes}_{\alpha_\lambda'}F)'$ in terms of certain bilinear forms. Let us define the Banach space formed by these applications.

\begin{definition}\rm
	We say that a bilinear form $B: E\times F\rightarrow \mathbb{K}$ is $\lambda$-integral if $B_L: E\widehat{\otimes}_{\alpha_\lambda'}F\rightarrow \mathbb{K}$ is continuous. We denote the space of all bilinear $\lambda$-integral forms by $\mathcal{B}_{I,\lambda}(E, F)$, endowed with the norm
	$$\left\|B \right\|_{{\cal B}_{I,\lambda}} = \inf\left\| \mu\right\| = \left\| B_L\right\|.$$ 
\end{definition}

From what we proved in the previous theorem we have
$\left(E\widehat{\otimes}_{\alpha_\lambda'}F \right)'\stackrel{1}{=} \mathcal{B}_{I,\lambda}(E, F)$.

\begin{remark}\rm
	a) A noteworthy fact: since $\varepsilon\leq \alpha_{\lambda}'$, one has $$\mathcal{B}_{I}(E,F)\subseteq \mathcal{B}_{I,\lambda}(E,F),$$for all Banach $E$ and $F$, i.e., every integral bilinear form (see \cite[Section 3.4]{R. Ryan}) is $\lambda$-integral.\\
	b) The only requirement on the norm $\alpha_\lambda'$ in Theorem \ref{teo.dual.gxy} is that it be a reasonable norm. For the cases where this norm is injective and if $M$ and $N$ are subspaces of $E$ and $F$, respectively, it follows from the Hahn-Banach Theorem that every $\lambda$-integral form defined on $M\times N$ can be extended to a $\lambda$-integral form defined on $E\times F$, with the same norm.
\end{remark}

The $\lambda$-integral forms satisfy a characteristic property of multi-ideals and the proof of this fact will be omitted for its simplicity.

\begin{proposition}
	Let $E,F,G$ and $H$ be Banach spaces, $T\in \mathcal{L}(G;E)$, $S\in \mathcal{L}(H;F)$ and $\lambda$ be a scalar sequence space. If $B\in \mathcal{B}_{I,\lambda}(E,F)$, then the bilinear form bilinear $B' = B \circ (T,S)$ is $\lambda$-integral and $$\left\|B' \right\|_{{\cal B}_{I,\lambda}}  \leq \left\|B \right\|_{{\cal B}_{I,\lambda}} \left\|T \right\| \left\| S\right\|.$$
\end{proposition}

\section{Tensor characterization for $(\ell_p^{\rm mid})^u(E)$}

Tensor characterizations for sequence spaces are common in tensor product theory and operator theory. Distinguished examples of what has been said are the characterizations $\ell_{p}\langle E \rangle \stackrel{1}{=} E \widehat{\otimes}_\pi \ell_p$ (cf. Example \ref{exnew}) and $\ell_p^u(E) = E \widehat{\otimes}_\varepsilon \ell_p$ (see \cite[Section 8.2]{A.Defant}), for $1 \le p < \infty$.

To establish our characterization for the space $(\ell_p^{\rm mid})^u(E)$ we first need to present the procedure $X \mapsto X^u$. This procedure was first studied in the thesis \cite{Lucas} (not published) and improved and detailed in \cite{Ariel1}. A few definitions and results from this last reference will be enough.

A sequence class $X$ is said to be \emph{finitely shrinking} if, regardless of the Banach space $E$, the sequence $(x_j)_{j=1}^\infty \in X(E)$ and $k \in \mathbb{N}$, it holds
\[(x_j)_{j\neq k} :=(x_1,\ldots,x_{k-1}, x_{k+1}, \ldots) \in X(E) \textrm{ and } \|(x_j)_{j\neq k}\|_{X(E)} \le  \|(x_j)_{j=1}^\infty\|_{X(E)}.\]

Suppose that $X$ is finitely shrinking. If $(x_j)_{j=1}^\infty \in X(E)$, then $(x_n, x_{n+1}, \ldots) \in X(E)$ for every $n \in \mathbb{N}$. For a Banach space $E$, we define
\[X^u(E) := \{(x_j)_{j=1} \in X(E): \lim_n \|(x_n, x_{n+1}, \ldots)\|_{X(E)} = 0\},\]
endowed with the norm $\|\cdot\|_{X(E)}$. It is proved in \cite[Proposition 4.5]{Ariel1} that the rule $E  \mapsto X^u(E)$ defines a sequence class and if $X$ is linearly stable, then so is $X^u(\cdot)$.

\begin{example}\rm a) For every $1 \le p < \infty$ it is immediate that $(\ell_p^w)^u = \ell_p^u$ and $(\ell_p)^u = \ell_p(\cdot)$. We also have $(\ell_\infty)^u = c_0(\cdot)$ and by \cite[Theorem 3.10]{Fourie} we get $(\ell_p\left\langle \cdot \right\rangle)^u = \ell_p\left\langle \cdot \right\rangle$.\\
b) For $1 \le p < \infty$, it is easy to check that the sequence class $\ell_p^{\rm mid}$ is finitely shrinking. Furthermore, we have $(\ell_p^{\rm mid})^u \neq \ell_p^{\rm mid}$ and $(\ell_p^{\rm mid})^u \neq \ell_p^u$ (\cite[Proposition 4.9]{Ariel1}).
\end{example}

For the results of this section, which comprise the next theorem and its immediate corollary, it is necessary to precisely specify the sequence classes $X$ and $Y$ involved in the definition of our appropriate $\alpha_\lambda^{ }$. Here it is also necessary to use the fact (see \eqref{defmid}): for $(x_j)_{j=1}^\infty \in E^\mathbb{N}$, we have 
\[(x_j)_{j=1}^\infty \in \ell_p^{\rm mid}(E) \ \ \Leftrightarrow \ \ ((\varphi_n(x_j))_{n=1}^\infty)_{j=1}^\infty \in \ell_p(\ell_p), \ \forall \ (\varphi_{n})_{n=1}^\infty \in \ell_p^w(E').\]

\begin{theorem}
	Let $E$ be a Banach space and $\lambda = \ell_p$, $1 \le  p<\infty$. Taking $X=\ell_p^w (\cdot)$ and $Y= \ell_{p^*}(\cdot)$ we obtain the isometric isomorphism
	\[(\ell_p^{\rm mid})^u(E) \stackrel{1}{=} E \widehat{\otimes}_{\alpha_{\ell_p}^{}} \ell_p.\]
\end{theorem}
\begin{proof}
	Considering the operator 
	\[J \colon E \times \ell_p \to (\ell_p^{\rm mid})^u(E)\ , \ (x,(a_k)^\infty_{k=1})\mapsto (a_kx)^\infty_{k=1},\]
	we have
	\begin{align*}
		\nonumber \|J((x,(a_k)^\infty_{k=1}))\|_{p, \rm mid} &= \sup_{(\varphi_n)^\infty_{n=1} \in B_{\ell_p^w(E')}}\left(\sum_{k=1}^\infty \sum_{n=1}^\infty |\varphi_n(a_kx)|^p \right)^{\frac{1}{p}}\\
		\nonumber&= \sup_{(\varphi_n)^\infty_{n=1} \in B_{\ell_p^w(E')}}\left(\sum_{k=1}^\infty|a_k|^p \sum_{n=1}^\infty |\varphi_n(x)|^p \right)^{\frac{1}{p}}\\
		\nonumber&= \|(a_k)^\infty_{k=1}\|_p\sup_{(\varphi_n)^\infty_{n=1} \in B_{\ell_p^w(E')}}\left(\sum_{n=1}^\infty |\varphi_n(x)|^p \right)^{\frac{1}{p}}\\
		&=\|(a_k)^\infty_{k=1}\|_p\|x\cdot e_1\|_{p, \rm mid}  = \|(a_k)^\infty_{k=1}\|_p\|x\| < \infty
	\end{align*}
	and as
	$$
	\lim_{n \to \infty}\|(a_kx)^\infty_{k=n}\|_{p, \rm mid} = \lim_{n \to \infty}\|(a_k)^\infty_{k=n}\|_p\|x\| = 0
	$$
	it follows that $J$ is well-defined. It is easy to prove that $J$ is  bilinear and its continuity follows from above calculus. We prove now that the linearization 
	\begin{equation*}
		J_L \colon E \otimes_{\alpha_{\ell_p}^{ }} \ell_p \to (\ell_p^{\rm mid})^u(E)\ ,\ 
		u = \sum_{i=1}^{m}x_i \otimes a_i \mapsto \left(\sum_{i=1}^{m}a_{i_k}x_i \right)^\infty_{k=1}
	\end{equation*}
	of $J$ is an isometry, where $\alpha_{\ell_p}^{ }$ is given by
	$$
	\alpha_{\ell_p}^{ }(u)= \sup\left\lbrace \left|\sum_{j=1}^{k}\sum_{n=1}^{\infty}\varphi_{n}(x_{j})\psi_{n}(y_{j}) \right|: (\varphi_{n})_{n=1}^{\infty}\in B_{\ell_p^w(E')}, (\psi_{n})_{n=1}^{\infty} \in B_{\ell_{p^*}(\ell_{p^*})} \right\rbrace.
	$$
	Indeed, from the well-definition of $J_L$ and the definition of $\ell_p^{\rm mid}(E)$, for any $u = \sum_{i=1}^{m}x_i \otimes a_i \in E \otimes_{\alpha_{\ell_p}^{ }} \ell_p$, we have
	$$
	\left( \left(\sum_{i=1}^{m}a_{i_k}\varphi_n(x_i) \right)^\infty_{n=1} \right)^\infty_{k=1} \in \ell_p(\ell_p), 
	$$
	for all $(\varphi_n)^\infty_{n=1} \in \ell_p^w(E')$. Thus,
	\begin{align*}
		\|J_L(u)\|_{p, \rm mid} &= 
		 \sup_{(\varphi_n)^\infty_{n=1} \in B_{\ell_p^w(E')}} \left\|\left( \left(\sum_{i=1}^{m}a_{i_k}\varphi_n(x_i) \right)^\infty_{n=1} \right)^\infty_{k=1} \right\|_p\\
		&=\sup_{(\varphi_n)^\infty_{n=1} \in B_{\ell_p^w(E')}} \left\|\left( \left(\sum_{i=1}^{m}a_{i_k}\varphi_n(x_i) \right)^\infty_{k=1} \right)^\infty_{n=1} \right\|_p\\
		&=\sup_{(\varphi_n)^\infty_{n=1} \in B_{\ell_p^w(E')}} \sup_{(b_n)^\infty_{n=1} \in B_{\ell_{p*}(\ell_{p*})}}\left|\sum_{n=1}^\infty b_n\left(\left(\sum_{i=1}^{m}a_{i_k}\varphi_n(x_i) \right)^\infty_{k=1} \right)  \right|\\
		&=\sup_{(\varphi_n)^\infty_{n=1} \in B_{\ell_p^w(E')}} \sup_{(b_n)^\infty_{n=1} \in B_{\ell_{p*}(\ell_{p*})}}\left|\sum_{n=1}^\infty\sum_{k=1}^\infty b_{n_k}\left(\sum_{i=1}^{m}a_{i_k}\varphi_n(x_i) \right)  \right|\\
		&=\sup_{(\varphi_n)^\infty_{n=1} \in B_{\ell_p^w(E')}} \sup_{(b_n)^\infty_{n=1} \in B_{\ell_{p*}(\ell_{p*})}}\left|\sum_{i=1}^{m}\sum_{n=1}^\infty  \varphi_n(x_i) \left(\sum_{k=1}^\infty  b_{n_k}a_{i_k}\right) \right|\\
		&=\sup_{(\varphi_n)^\infty_{n=1} \in B_{\ell_p^w(E')}} \sup_{(b_n)^\infty_{n=1} \in B_{\ell_{p*}(\ell_{p*})}}\left|\sum_{i=1}^{m}\sum_{n=1}^\infty  \varphi_n(x_i) b_n(a_i)   \right|
		= \alpha_{\ell_p}^{ }(u)
	\end{align*}
	and $J_L$ is an isometry and extends to an isometry $J_L : E \widehat{\otimes}_{\alpha_{\ell_p}^{ }} \ell_p \to (\ell_p^{\rm mid})^u(E)$. Furthermore, it is easy to verify that $J_L(E \otimes \ell_p)$ is a dense subspace in $(\ell_p^{\rm mid})^u(E)$ and, by the definition of completion, we can conclude that $J_L$ is an isometric isomorphism.
\end{proof}

As far as we know, the class of sequences $\ell_p^w$ is not a dual class for any class $X$ if $p=1$. For the other choices of $1 < p < \infty$, as $\ell_p^w$ is dual for $\ell_{p^*}\langle \cdot \rangle$ and $\ell_{p^*}(\cdot)$ is dual for $\ell_p(\cdot)$, the norm $\alpha_{\ell_p}^{ }$ can be interpreted as an $\alpha_{\ell_p}'$ norm and we immediately obtain the following

\begin{corollary} The isometric isomorphisms 
	\[((\ell_p^{\rm mid})^u(E))'\stackrel{1}{=}( E \widehat{\otimes}_{\alpha_{\ell_p}'} \ell_p)' \stackrel{1}{=} \mathcal{B}_{I,\ell_p}(E, F)\]
 hold for $1<p<\infty$. 
\end{corollary}

\medskip

\noindent Departamento de Ci\^{e}ncias Exatas,
Universidade Federal da Para\'iba,
58.297-000 -- Rio Tinto, Brazil. Departamento de Matem\'atica,
Universidade Federal da Para\'iba,
58.051-900 -- Jo\~ao Pessoa, Brazil. e-mails: jamilson@dcx.ufpb.br, jamilsonrc@gmail.com\\

\noindent Instituto Federal de Educa\c c\~ao, Ci\^encia e Tecnologia do Cear\'a (IFCE),
63.475-000 -- Jaguaribe, Brazil. e-mail: lucas.carvalho@ifce.edu.br\\

\noindent Departamento de Matem\'atica,
Universidade Federal da Para\'iba,
58.051-900 -- Jo\~ao Pessoa, Brazil.
e-mail: lfpinhosousa@gmail.com

\end{document}